\documentclass[a4paper, 11pt, nosumlimits, twoside] {amsart}
\title[Families over Fano varieties]{Families of canonically polarized manifolds over log Fano varieties}
\author{Daniel Lohmann}
\address{Daniel Lohmann\\
Albert-Ludwigs-Universit\"at Freiburg\\
Mathematisches Institut\\
Eckerstrasse 1\\
D-79104 Freiburg }
\email{\href{mailto:lohmann.daniel@arcor.de}{lohmann.daniel@arcor.de}}

\thanks{The author gratefully acknowledges partial support by the
  DFG-Forschergruppe 790 ``Classification of Algebraic Surfaces and Compact
  Complex Manifolds''.}

\usepackage{amsmath}
\usepackage{mathrsfs}  
\usepackage{stmaryrd} 
\usepackage{amssymb}
\usepackage{amsfonts}
\usepackage{amsthm}
\usepackage[latin1]{inputenc}
\usepackage{graphicx}
\usepackage[T1]{fontenc} 
\usepackage[arrow,curve,matrix]{xy}
\usepackage{verbatim} 
\usepackage{hyperref}
\usepackage{tikz}

\definecolor{darkgreen}{rgb}{0,.4,0}
\definecolor{darkblue}{rgb}{0,0,.5} 
\hypersetup{pdftex=true, colorlinks=true, breaklinks=true, linkcolor=darkblue, menucolor=darkblue, pagecolor=darkblue, urlcolor=darkblue, citecolor=darkgreen}
\newcommand*\switchenum{
\renewcommand*\theenumi{\alph{enumi}}
\renewcommand*\labelenumi{\theenumi)}
\renewcommand*\theenumii{\arabic{enumii}}
\renewcommand*\labelenumii{\theenumii)}
}

\numberwithin{subsection}{section}
\theoremstyle{plain}    
\newtheorem{thm}{Theorem}[section]
\newtheorem{defn}[thm]{Definition}

\numberwithin{equation}{thm}
\numberwithin{figure}{section}
\theoremstyle{plain}    
\newtheorem{cor}[thm]{Corollary}
\newtheorem{lem}[thm]{Lemma}

\newtheorem{ex}[thm]{Example}
\newtheorem{fact}[thm]{Fact}

\theoremstyle{plain}    
\newtheorem{prop}[thm]{Proposition}
\newtheorem{proclaim-special}[thm]{\specialthmname}

\theoremstyle{remark}
\newtheorem{rem}[thm]{Remark}
\newtheorem{obs}[thm]{Observation}

\newtheorem*{claim*}{Claim} 
\newtheorem{notation}[thm]{Notation}

\DeclareMathOperator{\codim}{codim}
\DeclareMathOperator{\NE}{NE}
\DeclareMathOperator{\NM}{NM}
\DeclareMathOperator{\Sym}{Sym}

\DeclareMathOperator{\cont}{cont}

\DeclareMathOperator{\Supp}{Supp}

\DeclareMathOperator{\Div}{Div}
\DeclareMathOperator{\WDiv}{WDiv}

\newcommand{\bN}{\mathbb N}
\newcommand{\bP}{\mathbb P}
\newcommand{\bQ}{\mathbb Q}
\newcommand{\bR}{\mathbb R}

\newcommand{\sA}{\mathcal A}
\newcommand{\sC}{\mathcal C}
\newcommand{\sD}{\mathcal D}

\newcommand{\fE}{\mathfrak E}
\newcommand{\fM}{\mathfrak M}
\newcommand{\fX}{\mathfrak X}

\newcommand{\ol}{\overline}
\newcommand{\rto}{\dashrightarrow}
\newcommand{\lrd}{\lfloor}
\newcommand{\rrd}{\rfloor}

\begin{document}

\maketitle

\begin{abstract}
Let $(X,D)$ be a dlt pair, where $X$ is a normal projective variety. We show that any smooth family of canonically polarized varieties over $X\setminus\Supp\lrd D\rrd$ is isotrivial if the divisor $-(K_X+D)$ is ample.
This result extends results of Viehweg-Zuo and Kebekus-Kov{\'a}cs. 

To prove this result we show that any extremal ray of the moving cone is generated by a family of curves, and these curves are contracted after a certain run of the minimal model program. In the log Fano case, this generalizes a theorem by Araujo from the klt to the dlt case.

In order to run the minimal model program, we have to switch to a $\bQ$-factorialization of $X$. As $\bQ$-factorializations are generally not unique, we use flops to pass from one $\bQ$-factorialization to another, proving the existence of a $\bQ$-factorialization suitable for our purposes.
\end{abstract}
\tableofcontents

\section{Introduction and main results}
\subsection{Introduction and main results}
Let $f^\circ : Y^\circ \to X^\circ$ be a smooth projective family of canonically polarized manifolds over a quasi projective manifold $X^\circ$ of dimension at most three. Kebekus and Kov{\'a}cs proved in \cite{KK10} and \cite{KK08} that the variation of the family is bounded by the Kodaira-Iitaka-dimension $\kappa(X^\circ)$. They distinguish two cases.
\begin{enumerate}
\item  If $\kappa(X^\circ)\geq 0$ then the variation is less than or equal to $\kappa(X^\circ)$. In this case the Kodaira-Iitaka-dimension is an upper bound for the variation.

\item If $\kappa(X^\circ)=-\infty$ then the variation of $f^\circ$ is not maximal.
\end{enumerate}
The upper bound given in the second case is generally optimal. For instance, any family of maximal variation over a variety $Z$ can be pulled back to a family over $Z\times\bP^1$. The base $Z\times\bP^1$ has negative Kodaira-Iitaka dimension, and the variation of the family is given by $\dim Z$. 

We ask if we obtain better results if we make additional assumptions. 
Clearly, if $X^\circ=\bP^1$, then Kebekus' and Kov{\'a}cs' result immediately implies that the family is isotrivial, see also \cite{Kov00}. This in turn implies that the family is necessarily isotrivial on rationally connected varieties. 
Therefore, isotriviality holds if $X^\circ$ is a Fano manifold, i.e., $X^\circ$ is projective and $-K_{X^\circ}$ ample.
 
In this paper, we will focus on \emph{log Fano varieties}, these are dlt pairs $(X,\Delta)$ with $-(K_X+\Delta)$ ample. 
The main result of this paper is stated in the following Theorem.

\begin{thm}[Isotriviality Theorem]\label{thm:isotrivial}
  Let $(X,\Delta)$ be a dlt pair where $\Delta$ is an effective $\bR$-divisor,
  where $-(K_X+\Delta)$ is $\bR$-ample, and $X$ is projective. Let $T\subset
  X$ be a subvariety of codimension greater or equal than two such that
  $X\setminus(T\cup\Supp\lrd\Delta\rrd)$ is smooth. Then any smooth family of
  canonically polarized varieties over $X\setminus(T\cup\Supp\lrd\Delta\rrd)$
  is isotrivial.
\end{thm}

It is still an open question if log Fano varieties are rationally connected by curves that intersect $\Delta$ in at most two points. Therefore, the short line of argument given above to show that families over Fano manifolds are isotrivial does not apply.

Instead, we will use Kebekus' and Kov{\'a}cs' result which asserts that any run of the minimal model program for $(X,\Delta)$ factorizes the moduli map birationally.
The following theorem, which is a generalization of a result by Araujo \cite[Theorem 1.1]{Ara08}, describes the different types of minimal model programs with scaling that can be run. 

\begin{thm}[Moving Cone Theorem]\label{thm:mcone} 
  Let $(X,\Delta)$ be a $\bQ$-factorial dlt pair, where $\Delta$ is an
  effective $\bR$-divisor and $X$ is projective. Let $R$ be an exposed ray of
  the cone $\ol\NM_1(X)+\ol\NE_1(X)_{K_X+\Delta \geq 0}$ that intersects
  $(K_X+\Delta)$ negatively.  Then there is an irreducible locally closed
  subset $H_R$ of the Hilbert scheme of curves on $X$ such that
  \begin{enumerate}
  \item each closed point of $H_R$ corresponds to a curve that generates $R$,
  \item for any closed subset $Z\subset X$ of $\codim_X(Z)\geq 2$, there is a
    non-empty open subset $H^Z_R$ of $H_R$ such that any curve that
    corresponds to a closed point of $H^Z_R$ avoids $Z$,\label{item:codim}
  \item there exists a run of the minimal model program with scaling that
    terminates with a Mori fiber Space
    \[
    \begin{xy}
      \xymatrix{
        X \ar@{ -->}[r]^{\lambda_R} & X_R \ar[d]^{\pi_R} \\
        & B_R
      }
    \end{xy}
    \]
    such that any closed point of $H_R$ corresponds to a curve that is
    contained in the open set $U\subset X$, where $\lambda_R$ is an
    isomorphism of $U$ onto its image. Moreover, the image of this curve via
    $\lambda_R$ is contained in a fiber of $\pi_R$.
   \end{enumerate}

\end{thm}

\subsection{Outline of paper} 
The Isotriviality Theorem \ref{thm:isotrivial} is a consequence of the Moving Cone Theorem \ref{thm:mcone}, thus we will first focus on the proof of the latter. In Section \ref{sec:dltmmp} we recall some facts and results of \cite{BCHM10} concerning the minimal model program, then we will explain the minimal model program with scaling. Finally, we will generalize some results for klt pairs to the dlt case. The proof of Theorem \ref{thm:mcone} is then given in Section \ref{sec:cone}.

In Section \ref{sec:nonfactorial} we will analyze different $\bQ$-factorializations of dlt pairs and show that a flop of a $\bQ$-factorialization yields a new one.
 We will use this result to construct for each effective Weil divisor $D$ on a log Fano dlt pair a $\bQ$-factorialization $(Y,\Delta_Y)$ such that the strict transform of $D$ is 
not numerically trivial on all $(K_Y+\Delta_Y)$-negative exposed rays of the cone $\ol\NM_1(Y)+\ol\NE_1(Y)_{K_Y+\Delta\geq0}$. 

In Section \ref{sec:isotrivial} we will finally prove the Isotriviality Theorem \ref{thm:isotrivial}. The proof is an induction on the dimension $n$ of the underlying variety. As part of the induction we prove Kebekus' and Kov{\'a}cs' result \cite[Theorem 1.2]{KK10} for varieties of negative Kodaira-Iitaka-dimension. 

Assuming Kebekus' and Kov{\'a}cs' result in dimension $n$, the moduli map induced by the family factors via any run of the minimal model program. An application of the Moving Cone Theorem \ref{thm:mcone} then describes the relevant minimal model programs in more detail. In particular, we will see that if $H_R$ is the set given in the Moving Cone Theorem, then the family restricted to a curve that corresponds to a general element of $H_R$ is isotrivial. The ampleness of $-(K_X+\Delta)$ implies that there are sufficiently many such rays. This finally implies the Isotriviality Theorem for $n$-dimensional varieties.

 On the other hand, the Isotriviality Theorem in dimension $n$, and the recently proven Bogomolov-Sommese vanishing for lc pairs \cite[Theorem 7.2]{GKKP10} imply  Kebekus' and Kov{\'a}cs' result for $(n+1)$-dimensional varieties of negative Kodaira-Iitaka-dimension. This completes the proof.

The last Section~\ref{sec:corollary} finally shows that the Isotriviality Theorem can be used to obtain a description of the moving cone of varieties that admit non-isotrivial families.

\subsection{Acknowledgements}
The results of this paper are part of the author's forthcoming Ph.D.~thesis. He would like to thank his supervisor Stefan Kebekus and his research group, especially Patrick Graf, Daniel Greb, and Sebastian Neumann.
He would also like to thank the research group's guests during the last two years for inspiring discussions.

The proof of the Moving Cone Theorem uses many methods of Carolina Araujo's proof for the klt case \cite[Theorem 1.1]{Ara08}. The structure of the cone $\ol\NM_1(X)+\ol\NE_1(X)_{K_X+\Delta\geq 0}$ has also been subject of work by Brian Lehmann \cite{Leh09} and, for Fano three- and fourfolds, by Sammy Barkowski \cite{Bar10}.

Questions concerning the variation of families have been discussed by many authors. Related results are shown in \cite[Theorem 0.1]{VZ02} by Eckart Viehweg and Kang Zuo, and by Stefan Kebekus and S{\'a}ndor Kov{\'a}cs in \cite{KK08}, \cite{KK08a}, and \cite{KK10}. Isotriviality criteria for families of canonically polarized varieties have a long history in Algebraic Geometry. We refer to \cite{Keb11} and \cite{Kov09} for a more complete overview.

\section{The minimal model program with scaling}\label{sec:dltmmp}
In this chapter we introduce the minimal model program with scaling and prove termination for the $\bQ$-factorial dlt case. This generalizes a result of \cite{BCHM10} from the klt to the dlt case.
 Although this generalization is probably well-known to experts, we will include a proof since  the methods used will be very useful to prove Theorem \ref{thm:mcone}.
\subsection{The standard minimal model program}
The reader who is not familiar with the classical minimal model program is referred to \cite{KM98}. Unless otherwise stated, a pair $(X,\Delta)$ consists of a projective normal variety $X$ and an $\bR$-divisor $\Delta$. We always demand that $K_X+\Delta$ is $\bR$-Cartier, but we do generally not assume that $X$ is $\bQ$-factorial.  
Moreover, we notice the following
\begin{rem}
 In \cite{KM98}, everything is stated for $\bQ$-divisors. Note that the relevant definitions of singularities can easily be extended to $\bR$-divisors. Moreover, using that $\bQ$ is dense in $\bR$ one can show that the Cone Theorem also holds for $\bQ$-factorial dlt pairs $(X,\Delta)$ with $\Delta$ being an $\bR$-divisor, see also Proposition \ref{prop:kltdlt}.

	A minimal model program may consist of infinitely many steps. If a minimal model program terminates, we call it a \emph{terminating minimal model program}.

	Each step of a minimal model program is either a divisorial contraction or a flip. If a minimal model program leads to a Mori fiber space $\pi:X_\lambda\to B$, then the map $\pi$ does not count as a step of the minimal model program.
\end{rem}

We will frequently use the following notation.
\begin{notation}\label{not:mmp}		
  Let $(X,\Delta)$ be a $\bQ$-factorial dlt pair, and let
\[
X=:X_0\stackrel{\varphi_1}{\rto} X_1 \stackrel{\varphi_2}\rto\dots\stackrel{\varphi_{n}}{\rto} X_n\stackrel{\varphi_{n+1}}{\rto}\dots
\]
be a (possibly infinite) run of the minimal model program. Let $i\in\bN$ such that the $i$th step $\varphi_{i}$ exists.
\begin{enumerate}
  \item Given an $\bR$-divisor $D$ on $X$, we set $D_0:=D$ and define an $\bR$-divisor $D_i$ on $X_i$ recursively as
\[
D_i:= (\varphi_{i})_*D_{i-1}.
\]
 \item We denote by $R_i\subset\ol\NE_1(X_{i-1})$ the $(K_{X_{i-1}}+\Delta_{i-1})$-negative extremal ray which is contracted or flipped by $\varphi_{i}$. If the minimal model program terminates with a Mori fiber space $X_m\to B$, we define $R_{m+1}$ analogously. 
\end{enumerate}
\end{notation}

\subsection{Pushforward and pullback of curves}
In the sequel we will sometimes have to take pushforward und pullback of numerical classes of 1-cycles. To define this, we use pullback and pushforward of classes of divisors and duality of the underlying vector spaces , see \cite[Chapter 3]{Bar08} and \cite[Chapter 4]{Ara08}.
\begin{defn}[Numerical pushforward and pullback of curves] \label{def:numerical} Let $f:X\rto Y$ be a birational map between $\bQ$-factorial varieties which is surjective in codimension one. 
Then we define the \emph{numerical pullback} and \emph{numerical pushforward}
\[
f^*:N_1(Y)\to N_1(X)\quad\text{and}\quad f_*:N_1(X)\to N_1(Y)
\]
as the dual maps of the pushforward and the pullback of divisors.
\end{defn}
\begin{rem}
If a curve is contained in the domain of the map, then the pushforward of its class coincides with the class of its cycle-theoretic pushforward, see \cite[Corollary 3.12]{Bar08}.

 On the other hand it is difficult to see what the pullback or pushforward of a curve is if it is contained in the indeterminacy locus of the underlying map. There are examples where the pullback of a curve behaves rather counterintuitively, see \cite[Examples 4.2 and 4.3]{Ara08}. 
\end{rem}
The definition above immediately implies the following identities.
\begin{prop}[Projection formulae]
  Let $f:X\rto Y$ be as in Definition~\ref{def:numerical}.
\begin{enumerate}
\item If $\gamma\in N_1(X)$ and $[D]\in N^1(Y)$, then $f_*\gamma\cdot [D] = \gamma\cdot f^*[D]$.
\item If $\gamma\in N_1(Y)$ and $[D]\in N^1(X)$, then $f^*\gamma\cdot [D] = \gamma\cdot f_*[D]$.
\end{enumerate}\qed
\end{prop}
\subsection{The minimal model program with scaling}

The existence of terminating minimal model programs can be proved if we take a given divisor into account.

\begin{defn}[Minimal model program with scaling]\label{def:smmp}
  Let $(X,\Delta)$ be a $\bQ$-factorial dlt pair, and let $H$ be an ample $\bR$-divisor such that $K_X+\Delta+ H$ is nef.
A \emph{(terminating) minimal model program with scaling of $H$} is a (terminating) minimal model program 
\[
X=:X_0\stackrel{\varphi_1}{\rto} X_1 \stackrel{\varphi_2}\rto\dots\stackrel{\varphi_{n}}{\rto} X_n\stackrel{\varphi_{n+1}}{\rto}\dots
\]
and a (finite) decreasing sequence of real numbers
\[
s_0\geq s_1\geq\dots\geq s_n\geq\dots\geq 0,
\]
such that for any $i$, where $R_{i}$ is defined, the following holds.
 \begin{enumerate}
\item The divisor $K_{X_{i-1}}+\Delta_{i-1}+s_{i-1}H_{i-1}$ is nef. \label{item:nef1}
\item \label{scont} The ray $R_i$ is contained in the hyperplane 
\[
(K_{X_{i-1}}+\Delta_{i-1}+s_{i-1}H_{i-1})^\perp\subset N_1(X).
\]

\item If the minimal model program terminates with a Mori fiber space $X_m\to B$, then $R_{m+1}\subset (K_{X_m}+\Delta_m+s_mH_m)^\perp$. 
\end{enumerate} 
We will denote a minimal model program with scaling of $H$ by the sequence of pairs $(\varphi_i, s_i)_i$. 

\end{defn}
\begin{rem}\label{rem:inf}
 An easy computation shows that the divisor $K_{X_i}+\Delta_i+s_{i-1}H_i$ is nef, see \cite[3.8]{Ara08}. Properties (\ref{item:nef1}) and (\ref{scont}) imply that $s_i$ is uniquely determined by the equation
\[
s_i= \inf\{s>0\,|\,K_{X_i}+\Delta_i+sH_i \text{ is nef}\}.
\]
We can therefore view a step of the minimal model program with scaling as
follows. The divisor $K_{X_i}+\Delta_i+s_{i-1}H_i$ is nef, and after scaling
$s$ down, it approaches the Mori cone and determines the ray $R_{i+1}$.  The
first step is visualized in the following picture.
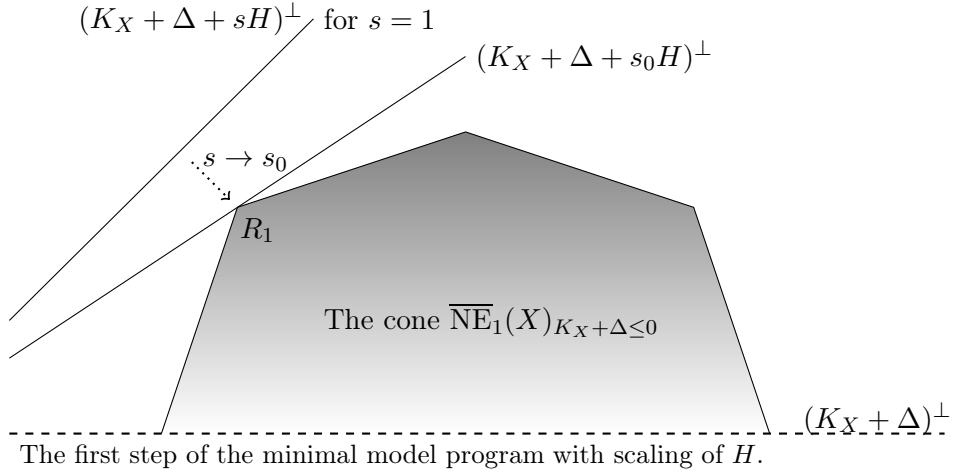
\begin{tikzpicture}

\path[shade] (2,0) -- (3,3) -- (6,4) -- (9,3) -- (10,0) --cycle;
\path[draw]      (2,0) -- (3,3) -- (6,4) -- (9,3) -- (10,0)  node[right] at(4,1.5){The cone $\ol\NE_1(X)_{K_X+\Delta\leq 0}$}; 
\draw[dashed, style= thick] (0,0) -- (12.3,0) node[right]at(10.3,0.2){$(K_X +\Delta)^\perp$}; 
\draw (0,1.5) -- ++(4,4)node[left]{$(K_X+\Delta+sH)^\perp$} node[right]{for $s=1$} ; 
\draw (0,1) -- ++(6,4)node[right]{$(K_X+\Delta+s_0H)^\perp$}; 
\path (2.9,2.7) node[right]{$R_1$};
\draw[->,dotted,thick] (2.4,3.6) node[right]{$s\to s_0$} --(2.9,3.1);

\path (0,-0.3) node[right]{\small The first step of the minimal model  program with scaling of $H$.};
\path (0,-1);
\end{tikzpicture}
\end{rem}
\begin{rem}
 It is a priori not clear that minimal model programs with scaling exist generally, even if flips are known to exist. Given $s_i$ as in Remark \ref{rem:inf}, we have to ensure the existence of an extremal ray $R\subset(K_{X_i}+\Delta_i+s_iH_i)^\perp$ that intersects $K_{X_i}+\Delta_i$ \emph{negatively}. The statement that for dlt pairs such a ray indeed exists is given in \cite[Lemma 3.1]{Bir10}. Hence we can always run a minimal model program with scaling, if flips exist. 
\end{rem}
For the klt case, termination of the minimal model program with scaling is stated in the following Theorem, see \cite[Corollary 1.3.3]{BCHM10} and \cite[Theorem 3.9]{Ara08}.
\begin{thm}[MMP with scaling for klt pairs]\label{thm:kltmmp}
Let $(X,\Delta)$ be a $\bQ$-factorial klt pair such that $K_X+\Delta$ is not pseudo-effective. Let $H$ be an effective ample $\bR$-divisor such that $K_X+\Delta+ H$ is nef and klt. Then any minimal model program with scaling of $H$ terminates with a Mori Fiber space. \qed
\end{thm}

\subsection{The minimal model program with scaling for dlt pairs}\label{sec:dlt}
In Theorem \ref{thm:dltmmp} we will show that Theorem \ref{thm:kltmmp} still holds for dlt pairs.
The proof uses that dlt pairs can be seen as the limit of klt pairs.
\subsubsection{dlt is the limit of klt}
The proof of the following Proposition \ref{prop:kmdlt} which is given in \cite{KM98} for $\bQ$-divisors does not directly apply to $\bR$-divisors. For that reason and for lack of an adequate reference for $\bR$-divisors, we provide short proofs of the results discussed in this section. A generalization of the following proposition for $\bR$-divisors is then given in Proposition \ref{prop:kltdlt}.

\begin{prop}[{\cite[Proposition 2.43]{KM98}}]\label{prop:kmdlt}
Assume that $(X,\Delta)$ is dlt ($\Delta$ a $\bQ$-divisor) and $X$ is quasi projective with ample divisor $H$. Let $\Delta_1$ be an effective $\bQ$-divisor (not necessarily $\bQ$-Cartier) such that $\Delta-\Delta_1$ is effective. Then there exists a rational number $c>0$ and an effective $\bQ$-divisor $D\sim_\bQ \Delta_1 +cH$ such that $(X,\Delta-\epsilon\Delta_1+\varepsilon D)$ is dlt for all rational numbers $0<\varepsilon\ll 1$. 

If $\Supp\Delta_1 = \Supp\Delta$, then $(X,\Delta-\varepsilon\Delta_1+\varepsilon D)$ is klt for all sufficiently small rational numbers $\varepsilon>0$.\qed
\end{prop}

\begin{lem}[See {\cite[Example 9.2.29]{Laz2}}]   \label{lem:ample}
Let $(X,\Delta)$ be a klt pair and $H$ an ample $\bR$-divisor. Then $H$ is $\bR$-linearly equivalent to an effective divisor $H'$ such that $(X,\Delta+H')$ is klt.
\end{lem}

\begin{proof}
  We first consider an ample $\bQ$-divisor $H$. Then for all sufficiently
  divisible $m\gg 0$, the divisor $mH$ is a very ample integral
  Cartier-divisor.  Let $\tilde H$ be a general member of $\left|mH\right|$,
  and set $H':=\frac{1}{m}\tilde H$. Since $m$ is chosen large, we have $\lrd
  \Delta+H'\rrd\leq 0$.  Moreover, it follows from \cite[Lemma 5.17]{KM98}
  that the discrepancy of $(X,\Delta+H')$ is still greater than $-1$.  This
  proves that $(X,\Delta+H')$ is klt.

  Since any ample $\bR$-divisor can be written as a positive linear
  combination of ample $\bQ$-divisors, it suffices without loss of generality
  to prove the assertion for an ample $\bR$-divisor of type $\lambda A$, where
  $\lambda\in\bR^+$ and $A$ is an ample $\bQ$-divisor. Choose a rational
  $l>\lambda$. As we have seen, there exists an ample $\bQ$-divisor $A'\sim
  lA$ such that $(X,\Delta + A')$ is klt. Clearly, $\frac{\lambda}{l}<1$ and
  $\frac{\lambda}{l}A'\sim \lambda A$. Therefore $\frac{\lambda}{l}A'$ has the
  required properties.
\end{proof}

\begin{prop}[Generalization of Proposition \ref{prop:kmdlt} for $\bR$-divisors] \label{prop:kltdlt}
Let $(X, \Delta)$ be a dlt pair and $H$ be an ample $\bR$-divisor. Then for any 
$\varepsilon>0$ there exists an effective $\bR$-divisor $\Delta_\varepsilon\sim_\bR \Delta+\varepsilon H$ such that the pair $(X,\Delta_\varepsilon)$ is klt.
\end{prop}
\begin{proof}

After rescaling of $H$ we can assume without loss of generality that $\varepsilon=1$. 
We first assume that $\Delta$ is a $\bQ$-divisor.
Since $H$ is not necessarily a $\bQ$-divisor, we write $H=H_1+H_2$ such that $H_1$ is an ample $\bQ$-divisor and $H_2$ is an ample $\bR$-divisor.

 There exists an $m\in\bN$ such that $mH_1$ is integral and Cartier, thus we may apply Proposition \ref{prop:kmdlt} for $\Delta_1=\Delta$ and $mH_1$. Accordingly there exists a rational number $c>0$ and an effective $\bQ$-divisor $D\sim_\bQ\Delta+ cmH_1$ such that for any sufficiently small $\varepsilon'>0$ the pair $(X,\Delta-\varepsilon'\Delta+\varepsilon' D)$ is klt.
In particular, $\Delta+\varepsilon'cm H_1$ is $\bR$-linearly equivalent to an effective $\bR$-divisor $\Delta_{H_1}$ such that $(X,\Delta_{H_1})$ is klt. 
By Lemma \ref{lem:ample}, we can replace  $\varepsilon' m c H_2$ by an $\bR$-linear equivalent effective divisor $H_3$ such that $(X,\Delta_{H_1}+H_3)$ is klt.
Note that 
\[\Delta_{H_1}+H_3 \sim_\bR \Delta+\varepsilon'mcH,
\]
thus another application of Lemma \ref{lem:ample}
 for $(1-\varepsilon' mc)H$ yields that $\Delta+H$ is $\bR$-linearly equivalent to an effective $\bR$-divisor $\Delta_H$ such that $(X,\Delta_H)$ is klt. This proves the claim if $\Delta$ is a $\bQ$-divisor.

Now we consider the general case where $\Delta$ is not necessarily a $\bQ$-divisor.
Because of the first part of the proof it suffices to find an effective $\bQ$-divisor $\Delta'$ such that
\begin{itemize}
  \item $(X,\Delta')$ is dlt,
  \item $H+\Delta-\Delta'$ is $\bR$-ample.
\end{itemize}
To prove the existence of $\Delta'$, we first write $\Delta$ as a positive linear combination 
\[
\Delta =\sum\limits_{i=1}^n r_iS_i,
\]
where $S_i$ are distinct prime Weil divisors and $r_i\in [0,1]$, for $i=1,\dots,n$. Consider $K_X\in\WDiv(X)$ as a fixed divisor which represents the canonical class and set 
\[
Q:=\left\{K_X+\sum\lambda_i S_i\,|\,\lambda_i\in [0,1]\right\}\subset\WDiv_\bR(X).
\]
Note that $Q$ is a rational polytope in $\WDiv_\bR(X)$ and consequently, the intersection $B:=Q\cap \Div_\bR(X)$ is a rational polytope as well. Moreover, $B$ is not empty because $K_X+\Delta\in B$.
Note that the property \emph{dlt} is an open property on $B$. More precisely, there is an open neighborhood $U\subset B$ of $K_X+\Delta$ such that the pair $(X,\Gamma)$ is dlt for any $K_X+\Gamma\in U$. Since ampleness is also an open property, we can assume that for any $K_X+\Gamma\in U$ the divisor $H+\Delta-\Gamma$ is ample. 

Since $B$ is a rational polytope, the set $B_\bQ:=Q\cap \Div_\bQ(X)$ is dense in $B$. Therefore, there exists $K_X+\Delta'\in U$ with $\Delta'$ being a $\bQ$-divisor. This finishes the proof.
\end{proof} 

\subsubsection{Termination of the minimal model program}
The following lemma shows that a variation of the boundary divisor $\Delta$ does not affect flips.
\begin{lem}[Rigidity of flips]\label{lem:flip}
Let $(X,\Delta)$ be a $\bQ$-factorial dlt pair. Assume that $R$ is a $(K_X+\Delta)$-negative extremal ray, and  that the contraction $f$ of $R$ is small. Let $D$ be an arbitrary $\bR$-divisor on $X$ such that $R$ is $(K_X+D)$-negative. 
If the $(X,\Delta)$-flip $\varphi$ of $f$ exists, then $\varphi$ is also the $(X,D)$-flip of $f$.
\end{lem}
\begin{proof}
Assume that any flip 
\[
\begin{xy}
\xymatrix{
X\ar@{ -->}[rr]^\varphi \ar[rd]_f && X^+ \ar[ld]^{f^+}\\
&Y& 
}
\end{xy}
\]
of $f$ exists. We have to show that $K_{X^+}+\varphi_*D$ is $f^+$-ample. Let $C^+\subset X^+$ be a curve which is contracted by $f^+$. Then it is shown in \cite[Lemma 4.13]{Bar08} that for the numerical pullback the following holds:
\[
-\varphi^*[C^+]\in R.
\]
Since $\varphi^*:N_1(X^+)\to N_1(X)$ is an isomorphism of vector spaces, the relative Picard number $\rho(X^+/Y)$ is one, and it suffices to show that $K_{X^+}+\varphi_*D$ intersects $C^+$ positively. This follows easily from the projection formula, thus $\varphi$ is a flip for both $(X,\Delta)$ and $(X,D)$.
\end{proof}
\begin{cor}\label{cor:dltmmp}
Let $(X,\Delta)$ be a $\bQ$-factorial dlt pair. Then any minimal model program (with scaling) can be run for $(X,\Delta)$.
\end{cor}
\begin{proof}
Since flips exist for klt pairs, see \cite[Corollary 1.4.1]{BCHM10} , Lemma \ref{lem:flip} and Proposition \ref{prop:kltdlt} imply the existence of flips for dlt pairs. This implies the assertion for arbitrary minimal model programs. It remains to show that for each step of a minimal model program with scaling there exists an extremal ray which can be contracted. This is shown in \cite[Lemma 3.10]{Bir10}.
\end{proof}
We are now able to generalize Theorem~\ref{thm:kltmmp}.
\begin{thm}[MMP with scaling for dlt pairs]\label{thm:dltmmp}
Let $(X,\Delta)$ be a $\bQ$-factorial dlt pair, and $H$ an ample $\bR$-divisor such that $K_X+\Delta+ H$ is nef. Assume that $K_X+\Delta$ is not pseudo-effective.
\begin{enumerate}
\item \label{item:diffmmp} Set $\sigma:=\inf\{s>0\,|\, K_X+\Delta+sH \text{ is pseudo-effective}\}$,
and let $0\leq\varepsilon_1,\varepsilon_2<\sigma$ be arbitrary real numbers. For $k\in\{1,2\}$, let $\Delta^k:=\Delta_{\varepsilon_k}$ be as in Proposition \ref{prop:kltdlt}, if $\varepsilon_k$ is positive, or set $\Delta^k:=\Delta$, if ${\varepsilon_k}=0$.

If $(\varphi_i, s_i)_i$ is a minimal model program with scaling for the pair $(X,\Delta^1)$, then $(\varphi_i, s_i+(\varepsilon_1-\varepsilon_2))_i$ is a minimal model program with scaling for the pair $(X,\Delta^2)$.
\item \label{item:terminates} Any minimal model program with scaling of $H$ can be run for the pair $(X,\Delta)$ and terminates.
\end{enumerate}
\end{thm}
\begin{proof}
It is shown in Corollary~\ref{cor:dltmmp} that the minimal model program with scaling can be run for dlt pairs. Item~(\ref{item:terminates}) is then a consequence of (\ref{item:diffmmp}) and Theorem~\ref{thm:kltmmp}. 

To show (\ref{item:diffmmp}), we first observe that for any $i$ the numerical equivalence 
\[
K_{X_i} +\Delta^1_i+s_i H_i \equiv K_{X_i}+\Delta^2_i +\left(s_i+(\varepsilon_1-\varepsilon_2)\right)H_i
\]
holds. In particular, the divisor $K_{X_i}+\Delta^2_i +\left(s_i+(\varepsilon_1-\varepsilon_2)\right)H_i$ is nef and numerically trivial on $R_{i+1}$. 
Moreover, if $R_{i+1}$ is $(K_{X_i}+\Delta^2_i)$-negative, then it follows from Lemma \ref{lem:flip} that a flip of $R_{i+1}$ does not depend on the numbers $\varepsilon_1,\varepsilon_2$.
It therefore remains to show that for any $i$ the following holds.
\switchenum{
\begin{enumerate}
\item \label{pf:positive} The number $s_i+(\varepsilon_1-\varepsilon_2)$ is positive,
  \item \label{pf:negative} the ray $R_{i+1}$ is $(K_{X_{i}}+\Delta^2_i)$-negative,
  \item \label{pf:terminates} if the first sequence terminates, then so does the second one.
\end{enumerate}
}
We first show that (\ref{pf:positive}) implies (\ref{pf:negative}), thus we assume that $s_i+(\varepsilon_1-\varepsilon_2)$ is positive for some $i$. Since $R_{i+1}$ is $(K_{X_i}+\Delta^1_i)$-negative and $(K_{X_i}+\Delta^1_i+s_iH_i)$-trivial, $R_{i+1}$ is $H_i$-positive. As we have seen before, $R_{i+1}$ is also $\left(K_{X_i}+\Delta^2_i +\left(s_i+\varepsilon_1-\varepsilon_2\right)H_i\right)$-trivial, and since $s_i+(\varepsilon_1-\varepsilon_2)$ is positive, we conclude (\ref{pf:negative}).

The next step is to show (\ref{pf:positive}) by induction on $i$. For $i=0$, it follows from Remark \ref{rem:inf} that 
\[
s_0=\inf\{s>0\,|\,K_X+\Delta_1+sH\text{ is nef}\}.
\]
In particular,
\[
s_0\geq\inf\{s>0\,|\,K_X+\Delta_1+sH\text{ is pseudo-effective}\}= \sigma-\varepsilon_1.
\] 
Therefore, $s_0+(\varepsilon_1-\varepsilon_2)\geq \sigma-\varepsilon_2$, which is positive by assumption.

For the induction step we assume that  $s_j+(\varepsilon_1-\varepsilon_2)$ is positive for each $j\leq i$, and we aim to show that $s_{i+1}+(\varepsilon_1-\varepsilon_2)$ is also positive. Assume this not the case. This immediately implies $\varepsilon_2>\varepsilon_1$, in particular $(X,\Delta^2)$ is klt. Moreover, the ray $R_{i+1}$ is $(K_{X_{i}}+\Delta^2_i)$-negative thus $\varphi_{i+1}$ is a step of a $(X,\Delta^2)$-minimal model program with scaling of $H$. We obtain the following nef $\bR$-divisors on $X_{i+1}$.
\begin{align*}
&K_{X_{i+1}}+\Delta^2_{i+1}+ \left(s_i+(\varepsilon_1-\varepsilon_2)\right)H_{i+1}\quad\text{and}\\
&K_{X_{i+1}}+\Delta^2_{i+1}+ \left(s_{i+1}+(\varepsilon_1-\varepsilon_2)\right)H_{i+1}
\end{align*}
Convexity of the nef cone implies that also $K_{X_{i+1}}+\Delta^2_{i+1}$ is nef, thus a run of the minimal model program with scaling for the pair $(X,\Delta^2)$ terminates with a minimal model, a contradiction to Theorem \ref{thm:kltmmp}.

It remains to show (\ref{pf:terminates}). We assume on the contrary that the first sequence terminates and the second one does not. This in particular implies that the first minimal model program terminates with a minimal model. Exchange $\varepsilon_1$ for $\varepsilon_2$, and we obtain a contradiction to (\ref{pf:positive}).
\end{proof}

\section{The moving cone of $\bQ$-factorial dlt pairs} \label{sec:cone}
The goal of this chapter is to prove Theorem \ref{thm:mcone}.
The proof is given in several steps.
We first analyze an arbitrary Mori fiber space and specify the curves we want to pull back. More precisely, we construct the following subvariety of the Hilbert scheme.

\begin{lem}\label{lem:moving}
Let $\lambda:X\rto X'$ be a birational map between normal projective varieties which is surjective in codimension one. Let $B$ be a variety with $\dim B<\dim X'$, and let $\pi:X'\to B$ a surjective morphism with connected fibers. 
Then there is an irreducible locally closed subvariety $H$ of the Hilbert scheme of curves on $X$ such that 
\begin{enumerate}
\item any closed point of $H$ corresponds to a moving curve that is contained in the open set where $\lambda$ is an isomorphism,
\item any closed point of $H$ corresponds to a curve $C$ whose image $\lambda(C)$ lies in a fiber of $\pi$, and
\item if $Z\subset X$ has codimension greater than or equal to two, then the set\[ H_Z :=\{p\in H\,|\, p \text{ corresponds to a curve that avoids }Z\}
\]
  is non-empty and open in $H$.
\end{enumerate}
\end{lem}

\begin{proof}
Let $U\subset X$ denote the set where $\lambda$ is an isomorphism onto its image $V:=\lambda(U)$. We aim to find a dominating family of curves that is entirely contained in $U$.

To this end, we first remark that the inverse $\lambda^{-1}$ does not contract any divisor, thus $\codim_{X'}(X'\setminus V)\geq 2$ holds. Therefore, if $F$ is a general fiber of $\pi$, then $\codim_F(F\setminus V)\geq 2$, as well. 
Let $k$ be the relative dimension of $X'$ over $B$, and pick $k-1$ very ample divisors $H_1,\dots,H_{k-1}$ on $X'$. If $D_1,\dots,D_{k-1}$ are general members of the corresponding linear systems $\left|H_1\right|,\dots,\left|H_{k-1}\right|$ then the intersection $F\cap D_1\cap\dots\cap D_{k-1}\subset F$ is an irreducible smooth curve that avoids $X'\setminus V$. 
We conclude that there is an open subset $U\subset B\times \left|H_1\right|\times\dots\times\left|H_{k-1}\right|$ such that for $(b,D_1,\dots,D_{k-1})\in U$  the intersection $\pi^{-1}(b)\cap D_1\cap\dots\cap D_{k-1}$ is a smooth curve.
This defines a family of curves that are entirely contained in $V$. Moreover, if $Z'$ is any subvariety of $X'$ of codimension greater than or equal to two then the general member of this family avoids $Z'$.

Via $\lambda$ we obtain the required family of curves on $X$ which in turn defines the subset $H$ of the Hilbert scheme. Moreover, if $Z\subset X$ has codimension greater than or equal to two, then $Z':=\lambda(U\cap Z)\subset X'$ has codimension greater than or equal to two as well. Thus a general point of $H$ corresponds to a curve that avoids $Z$.  
\end{proof}

\begin{cor}
If a minimal model program leads to a Mori fiber space, then the numerical pullback of any curve on a fiber of the Mori fiber space is a moving class. 
\end{cor}

\begin{proof}
Note that a minimal model program which leads to a Mori fiber space satisfies the condition of Lemma \ref{lem:moving}.
  Let $X'\to B$ denote the Mori fiber space, then the relative Picard number $\rho(X'/B)$ is one. Thus all classes of curves in fibers are numerically proportional in $X'$, and Lemma \ref{lem:moving} shows that these classes are moving. 
\end{proof}
The next step in the proof of Theorem \ref{thm:mcone} is the construction of a divisor suitable for running the minimal model program with scaling. This will be done in the following lemma which is strongly related to \cite[Lemma 4.3]{Leh09}. A similar statement is also given in \cite[Proof of Theorem 1.1]{Ara08}. 
\begin{lem}\label{lem:bounding}
  Let $(X,\Delta)$ be a $\bQ$-factorial dlt pair and let
\[R\subset \ol\NM_1(X)+\ol\NE_1(X)_{K_X+\Delta\geq 0}
\]
 be a $(K_X+\Delta)$-negative exposed ray.
Then there is an $\bR$-ample $\bR$-divisor $H$ such that for $\sigma:=\inf\{s>0\,|\, K_X+\Delta+sH\in\ol\NE^1(X)\}$ the following holds.
\begin{enumerate}
\item The divisor $ K_X+\Delta+H$ is nef.\label{item:nef}
\item $(K_X+\Delta+\sigma H)^\perp\cap(\ol\NM_1(X)+\ol\NE_1(X)_{K_X+\Delta\geq 0}) = R$. \label{item:bounding}
\item $(K_X+\Delta+s H)^\perp\cap(\ol\NM_1(X)+\ol\NE_1(X)_{K_X+\Delta\geq 0}) = 0$, if $s>\sigma$. \label{item:unique}
\end{enumerate}
\end{lem}
\begin{rem}[Picture]
The assertion of the previous lemma can be visualized in the following picture which shows the $(K_X+\Delta)$-negative part of the cones.
\begin{tikzpicture}
\path[shade] (2,0) -- (3,3) -- (6,4) -- (9,3) -- (10,0) --cycle;
\path[draw]      (2,0) -- (3,3) -- (6,4) -- (9,3) -- (10,0)  node[right] at(6.1,2.9){$\ol\NE_1(X)_{K_X+\Delta\leq 0}$}; 

\path[draw, fill=orange] (2,0) -- (4,2) -- (7,2.5) -- (8,2) -- (10,0) --cycle node[right] at(4,0.5){$\ol\NM_1(X)+\ol\NE_1(X)_{K_X+\Delta\geq 0}$}; 

\draw[dashed, style= thick] (0,0) -- (12.3,0) node[right]at(10.3,0.2){$(K_X +\Delta)^\perp$}; 
\draw (0,1.5) -- ++(4,4)node[right]{$(K_X+\Delta+H)^\perp$} ; 
\path (3.9,1.7) node[right]{Exposed ray $R$}; 
\draw[->,dotted,thick] (2.4,3.7) node[right]{Scaling of $H$} --(3.6,2.2); 

\draw (1,0.5) -- ++(7.5,3.75) node[right]{$(K_X+\Delta+\sigma H)^\perp$}; 

\path (-0.2,-0.3) node[right]{\small The picture suggests that the minimal model  program with scaling of $H$} (-0.2,-0.6)node[right]{\small terminates with the contraction of $R$.};
\path (0,-1);
\end{tikzpicture}
\end{rem}
\begin{proof}
We start with the construction of $H$. 
By definition of \emph{exposed} there exists an $\bR$-divisor $D$ such that 
\[
R= D^\perp\cap(\ol\NM_1(X)+\ol\NE_1(X)_{K_X+\Delta\geq 0}),
\]
and $D$ is non-negative on $\ol\NM_1(X)+\ol\NE_1(X)_{K_X+\Delta\geq 0}$.
We claim that there is an $a>0$ such that $D-a(K_X+\Delta)$ is an ample $\bR$-divisor. %
If $-(K_X+\Delta)$ is ample, we can take any sufficiently large $a$. Thus we may assume without loss of generality that $-(K_X+\Delta)$ is not ample.  
Since $D$ and $K_X+\Delta$, considered as forms on $N_1(X)$,  have no common zeros in $\ol\NE_1(X)\setminus\{0\}$, there exists a hyperplane $Z\subset N_1(X)$ such that 
\[
\left(D^\perp\cap (K_X+\Delta)^\perp\right)\subset Z\quad\text{and}\quad Z\cap \ol\NE_1(X)=\{0\}.
\]
It follows from basic linear algebra that 
there exist $b,c\in\bR$ such that $Z=(bD + c(K_X+\Delta))^\perp$, i.e., for any $x\in \ol\NE_1(X)\setminus\{0\}$ the inequality $(bD + c(K_X+\Delta))\cdot x \neq 0$ holds. This inequality still holds if we slightly vary $b$ and $c$, thus we may assume that both $b$ and $c$ are not zero. We set $a:=-\frac{c}{b}$, and it remains to show that the resulting divisor is ample and that $a$ is positive.
Since $-(K_X+\Delta)$ is not ample, there exists $w\in \ol\NE_1(X)\setminus\{0\}$ intersecting $K_X+\Delta$ trivially. Thus we have $(D-a(K_X+\Delta))\cdot w = D\cdot w >0$, by the choice of $D$. Since the cone $\ol\NE_1(X)$ is connected, the divisor $D-a(K_X+\Delta)$ intersects any element of $\ol\NE_1(X)\setminus\{0\}$ positively, and Kleiman's ampleness criterion implies that the divisor is ample.
To see that $a$ is positive we consider the intersection product of $D-a(K_X+\Delta)$ with a generator $z$ of $R$. Since this is positive, $a$ is positive and the claim follows.

To finish the construction of $H$, we choose $l>0$ such that $K_X+\Delta + l(D-a(K_X+\Delta))$ is nef, and set 
\[H:=l(D-a(K_X+\Delta)).
\] 

It remains to show that $H$ has the required properties.
Property (\ref{item:nef}) follows immediately from the construction of $H$. To show Property (\ref{item:bounding}), we first observe that $D$ is numerically proportional to $K_X+\Delta+\frac{1}{al}H$. By \cite[Theorem 2.2]{BDPP04}, the cones $\ol\NM_1(X)$ and $\ol\NE^1(X)$ are dual. Consequently, the divisor $D$ is pseudo-effective, in particular $\sigma\leq \frac{1}{al}$. Moreover, $K_X+\Delta + sH$ intersects any generator of $R$ negatively for any  $s<\frac{1}{al}$. Therefore $\sigma=\frac{1}{al}$ and $D$ is numerically proportional to $K_X+\Delta+\sigma H$, hence
\begin{align*}
(K_X+\Delta+\sigma H)^\perp&\cap(\ol\NM_1(X)+\ol\NE_1(X)_{K_X+\Delta\geq 0})\\
=D^\perp&\cap(\ol\NM_1(X)+\ol\NE_1(X)_{K_X+\Delta\geq 0}) =R,
\end{align*}
 as required. 

To prove the last Property (\ref{item:unique}), recall that $H$ is ample. This immediately implies that for any $s>0$ and $\gamma\in\ol\NE_1(X)_{K_X+\Delta\geq 0}$ the intersection product $(K_X+\Delta+sH)\cdot\gamma$ is positive. Moreover, for any $s>\sigma$ the divisor $K_X+\Delta+sH= K_X+\Delta+\sigma H + (s-\sigma)H$ is big, thus it intersects any $\gamma\in\ol\NM_1(X)$ positively. 
\end{proof}
With the previous lemmas at hand, we are now able to prove the Moving Cone Theorem \ref{thm:mcone}.   
\begin{proof}[Proof of Theorem \ref{thm:mcone}]
Let $(X,\Delta)$ and $R$ be as in Lemma~\ref{lem:bounding}. We apply this lemma and obtain an $\bR$-ample $\bR$-divisor $H$ and positive number $\sigma$ that satisfy properties (\ref{item:nef}), (\ref{item:bounding}), (\ref{item:unique}). The existence of $R$ implies that $K_X+\Delta$ is not pseudo-effective, see \cite[Theorem 2.2]{BDPP04}. By Theorem~\ref{thm:dltmmp} we obtain a terminating minimal model program with scaling of $H$ which we denote $(\varphi_i,s_i)_{i\in I}$. 
By Proposition \ref{prop:kltdlt}, there exists for any $0< \varepsilon < \sigma$ an $\bR$-divisor $\Delta_\varepsilon \equiv \Delta+\varepsilon H$ such that $(X,\Delta_\varepsilon)$ is klt.
It follows from Theorem \ref{thm:dltmmp} that the sequence $(\varphi_i,s_i-\varepsilon)_{i\in I}$ is a minimal model program with scaling of $H$ for the pair $(X,\Delta_\varepsilon)$, and that both minimal model programs terminate with a Mori fiber space, say $\pi:X_l\to B$. Denote by $\lambda$ the composition of all $\varphi_i$, $i\in I$, then we obtain the following diagram 
\[
\begin{xy}
\xymatrix{
           X \ar@{ -->}[r]^{\lambda} & X_l\ar[d]^\pi \\
           & B.
            }
\end{xy}
\]
The family of curves constructed in Lemma \ref{lem:moving} gives the required subset $H_R$ of the Hilbert scheme. It remains to show that the class $\gamma$ of a curve corresponding to a closed point of $H_R$ generates $R$. Since $\gamma$ is moving and because of Property (\ref{item:bounding}) of Lemma \ref{lem:bounding}, it suffices to prove that  the equality 
\[ 
(K_X+\Delta+\sigma H)\cdot\gamma =0
\]
holds.

To this end, we consider the decreasing sequence of positive numbers
\begin{align*}
 s_1-\varepsilon\geq s_2-\varepsilon\geq \cdots \geq s_l-\varepsilon \geq 0.
\end{align*}
Since the inequality $s_l-\varepsilon \geq 0$ holds for all $\varepsilon\in [0,\sigma)$, we obtain $s_l\geq \sigma$.
To show $s_l\leq \sigma$, we note that if $C$ is any curve on a general fiber of $\pi$, then the class $\gamma$ is numerically proportional to $\lambda^*([C])$. Therefore
\begin{align*}
0 &=(K_{X_l}+\Delta_l+s_l{\lambda}_*H)\cdot C\\
  &= (K_X+\Delta+s_l H)\cdot\gamma.
\end{align*}
Consequently, Property (\ref{item:unique}) of Lemma \ref{lem:bounding} implies $s_l = \sigma$. We now apply Property (\ref{item:bounding}) of Lemma \ref{lem:bounding} again, which implies that $R$ is generated by $\gamma$.
\end{proof}

\section{$\bQ$-factorializations of dlt pairs}\label{sec:nonfactorial}

\subsection{$\bQ$-factorialization}
If $(X,\Delta)$ is a dlt pair where $X$ is \emph{not} $\bQ$-factorial, then we cannot apply Theorem \ref{thm:mcone}. To overcome this difficulty, we aim to replace $X$ with a small, $\bQ$-factorial modification. 

\begin{defn}[$\bQ$-factorialization]
Let $X$ be a normal projective variety. A \emph{$\bQ$-factorialization of $X$} is a proper birational morphism $f:Y\to X$ where $Y$ is a normal projective $\bQ$-factorial variety and the exceptional set of $f$ has codimension greater than or equal to two in $Y$.
\end{defn}
\begin{ex}
Let $(Y,\Delta)$ be a $\bQ$-factorial dlt pair. Assume that there is a $(K_Y+\Delta)$-negative extremal ray $R$ of the cone $\ol\NE_1(Y)$ whose associated contraction map $\cont_R: Y\to X$ is small. Then $X$ is not $\bQ$-factorial and $\cont_R:Y\to X$ is a $\bQ$-factorialization of $X$. 
\end{ex}

The existence of $\bQ$-factorializations of dlt pairs is a result of \cite{BCHM10}. 

\begin{prop}[{\cite[Corollary 1.4.3]{BCHM10}}]\label{prop:factorial}
Let $(X,\Delta)$ be a log canonical pair and let $f:W\to X$ be a log resolution. Suppose that there is a divisor $\Delta_0$ such that $K_X+\Delta_0$ is klt. Let $\fE$ be any set of valuations of $f$-exceptional divisors which satisfies the following two properties:
\begin{enumerate}
\item $\fE$ contains only valuations of log discrepancy at most one, and
\item the centre of every valuation of log discrepancy one in $\fE$ does not contain any non-klt centres.
\end{enumerate}
Then we may find a proper birational morphism $\pi:Y\to X$, such that $Y$ is $\bQ$-factorial and the exceptional divisors of $\pi$ correspond to the elements of $\fE$.
\qed
\end{prop}
We state the explicit result for dlt pairs in the following corollary. For klt pairs this is also explained in the discussion after the formulation of Corollary 1.4.3 in \cite[p.9]{BCHM10}. 

\begin{cor}[Existence of $\bQ$-factorializations]\label{cor:factorial}
Let $(X,\Delta)$ be a dlt (resp. klt) pair. Then a $\bQ$-factorialization of $X$ exists. Moreover, if $f:Y\to X$ is an arbitrary $\bQ$-factorialization of $X$, and $\Delta_Y:=f^{-1}_*\Delta$ is the strict transform of $\Delta$,  then the pair $(Y,\Delta_Y)$ is dlt (resp. klt).
\end{cor}

\begin{proof}
If $(X,\Delta)$ is dlt, then we may apply Proposition \ref{prop:kltdlt} and find a divisor $\Delta'$ such that $(X,\Delta')$ is klt. Therefore, the existence of a $\bQ$-factorialization follows from Proposition \ref{prop:factorial}, if we set $\fE=\emptyset$. 

Now let $f:Y\to X$ be an arbitrary $\bQ$-factorialization, and let $\Delta_Y$ be the strict transform of $\Delta$. Note that $f$ is small, thus the equalities 
\[
f^*(K_X+\Delta) = K_Y+\Delta_Y\quad\text{and}\quad f_*\Delta_Y=\Delta_X
\]
 hold. Moreover, the coefficients of $\Delta_Y$ are exactly the coefficients of $\Delta$, hence $\lrd \Delta \rrd =0 $ iff $\lrd \Delta_Y\rrd=0$. A straightforward calculation yields that the discrepancies of $(X,\Delta)$ and $(Y,\Delta_Y)$ are equal, which in turn implies that $(Y,\Delta_Y)$ is klt if $(X,\Delta)$ is klt; see also \cite[Lemma 2.30]{KM98}.

To show that the property \emph{dlt} is preserved, recall its definition, \cite[Definition 2.37]{KM98}. According to this, it remains to prove that the strict transform of an snc divisor on the smooth locus $U$ of $X$ is an snc divisor on $f^{-1}(U)\subset Y$. We even claim that $f|_{f^{-1}(U)}$ is an isomorphism. Indeed, if $x\in U$ is a point where the inverse map $f^{-1}$ is not regular, then \cite[Chapter II.4, Theorem 2]{Sha1} immediately implies that $f$ contracts a divisor. This contradicts the assumption that $f$ does not contract divisors. 
\end{proof}

\begin{notation}
Given a dlt pair $(X,\Delta)$ and a $\bQ$-factorialization $f:Y\to X$, we will denote by  $\Delta_Y$ the strict transform of $\Delta$ as defined in Corollary \ref{cor:factorial}. 
\end{notation}

\begin{rem}
In fact, $\bQ$-factorializations of a given variety are generally not unique. As we will see in Section \ref{sec:flops}, any log flop of a $\bQ$-factorialization yields a new $\bQ$-factorialization.
\end{rem}

\subsection{$\bQ$-factorializations of log Fano varieties}
We consider dlt pairs $(X,\Delta)$ with $-(K_X+\Delta)$ ample. Unfortunately, if $f:Y\to X$ is a $\bQ$-factorialization, then the divisor $-(K_Y+\Delta_Y)=-f^*(K_Y+\Delta)$ is generally not ample, unless $f$ is the identity.
Nevertheless, the following lemmas hold.
\begin{lem} \label{lem:bignef}
Let $(X,\Delta)$ be a dlt pair with $-(K_X+\Delta)$ ample, and let $f:Y\to X$ be a $\bQ$-factorialization of $X$. Then the divisor $-(K_Y+\Delta_Y)$ is big and nef.
\end{lem}

\begin{proof} 
Since $-(K_X+\Delta)$ is ample, it is in particular big and nef. The pullback of a big and nef $\bR$-Cartier divisor via a birational morphism is again big and nef.
\end{proof}

\begin{lem} \label{lem:kltmds}
  Let $(X,\Delta)$ be a $\bQ$-factorial klt pair such that $-(K_X+\Delta)$ is big and nef.
Then the cones $\ol\NM_1(X)$ and $\ol\NE_1(X)$ are rational polyhedrons. Moreover, for any divisor $D$ any minimal model program for the pair $(X,D)$ can be run and terminates.
\end{lem}
\begin{proof}
Recall from \cite[Theorem 2.2]{BDPP04} that a divisor is big if and only if it intersects any $\gamma\in \ol\NM_1(X)\setminus\{0\}$ positively. Hence, the cones $\ol\NM_1(X)\setminus\{0\}$ and $\ol\NE_1(X)_{K_X+\Delta=0}\setminus\{0\}$ are disjoint, and by convexity there exists an $\bR$-divisor $B$ that separates these cones, i.e., 
\begin{align*}
 \ol\NM_1(X)\setminus\{0\}&\subset N_1(X)_{B>0}, \quad\text{and}\\
  \ol\NE_1(X)_{K_X+\Delta=0}\setminus\{0\}&\subset N_1(X)_{B<0}.
\end{align*}
In particular, the divisor $B$ is big.

We claim that for sufficiently small $\varepsilon >0$ the pair $(X,\Delta+\varepsilon B)$ is still klt and the divisor $-(K_X+\Delta+\varepsilon B)$ is ample. To prove the claim we first note that for any sufficiently small $\varepsilon>0$ the pair $(X,\Delta+\varepsilon B)$ is klt, see \cite[Corollary 2.35(2)]{KM98}.
To show that $-(K_X+\Delta+\varepsilon B)$ is ample for $0<\varepsilon\ll 1$, we use Kleiman's ampleness criterion. According to this, we must show that the intersection product with any class $\gamma \in \ol\NE_1(X)\setminus\{0\}$ is positive. This is obviously true for $\gamma \in \ol\NE_1(X)_{B<0}\setminus\{0\}$, thus it remains to show that the intersection product with any class $\gamma \in \ol\NE_1(X)_{B\geq 0}\setminus\{0\}$ is positive.
 Let $H\subset N_1(X)_\bR\setminus\{0\}$ be an affine hyperplane such that its intersection with the Mori cone is a \emph{cross section}, i.e.,
\[\emptyset\neq \ol\NE_1(X)|_H:= H\cap\ol\NE_1(X)
\]
 is compact, and  
\[
\ol\NE_1(X)=\bR^{\geq 0} \cdot\ol\NE_1(X)|_H.
\] It suffices to show that $-(K_X+\Delta+\varepsilon B)$ intersects any class $\gamma\in \ol\NE_1(X)|_{H, B\geq 0}$ positively. Since $\ol\NE_1(X)|_{H, B\geq 0}$ is compact, the continuous function
\begin{align*}
\ol\NE_1(X)|_{H, B\geq 0}&\to\bR\\
 \gamma&\mapsto -(K_X+\Delta+\varepsilon B)\cdot\gamma
\end{align*}
has a global minimum $m_\varepsilon\in\bR$. This minimum depends continuously on $\varepsilon$ and is positive for $\varepsilon=0$. Consequently, the claim follows.

 The Cone Theorem implies that $\ol\NE_1(X)$ is a rational polyhedron, and the assertion for $\ol\NM_1(X)$ is proved in \cite[Corollary 1.3.5]{BCHM10}.
To show that for any divisor $D$ the minimal model program terminates, we apply \cite[Corollary 1.3.2]{BCHM10} to $(X,\Delta+\varepsilon B)$. According to this, the variety $X$ is a Mori dream space (see \cite[Definition 1.10]{HK00} for the definition), and it follows from \cite[Proposition 1.11]{HK00} that the minimal model program can be run for any divisor and terminates.
\end{proof}

\begin{cor}\label{prop:mds}
Let $(X,\Delta)$ be a dlt pair with $-(K_X+\Delta)$ ample, and let $f:Y\to X$ be any $\bQ$-factorialization of $X$. Then the cones $\ol\NE_1(Y)$ and $\ol\NM_1(Y)$ are rational polyhedrons and for any divisor the minimal model program can be run and terminates.
\end{cor}

\begin{proof}
By Lemma \ref{lem:kltmds} it suffices to show that there is a divisor $\Delta'$ on $Y$ such that $(Y,\Delta')$ is klt and $-(K_Y+\Delta')$ is big and nef.
In order to prove the existence of $\Delta'$ we first pick an ample divisor $H$ on $X$. It follows from Proposition \ref{prop:kltdlt} that for any $\varepsilon>0$ the divisor $\Delta+ \varepsilon H$ is $\bR$-linearly equivalent to a divisor $\Delta_\varepsilon$ such that $(X,\Delta_\varepsilon)$ is klt. Moreover, if $\varepsilon$ is sufficiently small then $-(K_X+\Delta_\varepsilon)$ is still ample. By Corollary \ref{cor:factorial} the pair $\left(Y,f^{-1}_*(\Delta_\varepsilon)\right)$ is klt, and Lemma \ref{lem:bignef} implies that $-(K_Y+\Delta')$ is big and nef.
\end{proof}
\subsection{Log flops of $\bQ$-factorializations}\label{sec:flops}
One main step in the proof of the Isotriviality Theorem \ref{thm:isotrivial} is to find a certain exposed moving ray which intersects a given pseudo-effective divisor $D$ non-trivially. This is not a big problem if the pair $(X,\Delta)$ is $\bQ$-factorial and log Fano. However, if we drop the assumption that $X$ is $\bQ$-factorial, then we have to switch over to a $\bQ$-factorialization $f:Y\to X$ which is generally not log Fano, as we have seen. Indeed, it could happen in this situation that the set of exposed moving rays is entirely contained in the hyperplane $(f^{-1}_*D)^\perp$ in $N_1(Y)$.
 
To prove the Isotriviality Theorem \ref{thm:isotrivial} in the non-$\bQ$-factorial case we have to find the \emph{right} $\bQ$-factorialization. We will see that a certain class of birational maps gives us new $\bQ$-factorializations. These \emph{log flops} are strongly connected to flips. 
\begin{defn}[Log flops, see {\cite[Conjecture 11.3.3]{Mat02}}]
  Let $(X,\Delta)$ be a dlt pair. A \emph{flopping contraction} is a proper birational morphism $f:X\to Y$ to a normal variety $Y$ such that the exceptional set has codimension at least two in $X$ and $K_X+\Delta$ is numerically $f$-trivial.

Assume that there exists an $\bR$-Cartier divisor $D$  on $X$ such that $-(K_X+\Delta+D)$ is $f$-ample, and the $(K_X+\Delta+D)$-flip of $f$ exists. Then this flip is also called the \emph{$D$-log flop of $f$} or \emph{log flop} for short. 
\end{defn}

\begin{rem}
If $\Delta=0$, a log flop is a \emph{flop}, see \cite[Definition 6.10]{KM98}.
\end{rem}

\begin{lem}[Existence of log flops on $\bQ$-factorializations]\label{lem:flop}
 Let $(X,\Delta)$ be a log Fano dlt pair with $\bQ$-factorialization $(Y,\Delta_Y)$. Let $D$ be an arbitrary $\bR$-divisor on $Y$, and let $F\subset\ol\NE_1(Y)_{K_Y+\Delta_Y=0}$ be an extremal face that is contained in ${D<0}$. Then
\begin{enumerate}
 \item \label{lem:flop1}the contraction $g:Y\to Z$ of $F$ exists and factorizes the $\bQ$-factorialization map $f:Y\to X$, and
 \item \label{lem:flop2}the $D$-log flop of $F$ exists and is another $\bQ$-factorialization of $X$. 
\end{enumerate}
\end{lem}

\begin{proof}
By Corollary \ref{prop:mds} the minimal model program for the pair $(Y,D)$ is well-defined, in particular the contraction $g:Y\to Z$ of $F$ exists. 
   To prove that $g$ is small, we note that the map $f:Y\to X$ is the contraction of the extremal face $G:=\ol\NE_1(Y)\cap (K_Y+\Delta_Y)^\perp$. Indeed, it is easy to see that a curve $C$ is contracted by $f$ iff it intersects $K_Y+\Delta_Y$ trivially. Since this is a small contraction and $F\subset G$ is a subface, any curve that is contracted by $g$ is also contracted by $f$. Therefore, the exceptional set of $g$ has codimension at least two, hence $g$ is a small contraction.
It remains to show that $g$ factorizes $f$. We have already seen that $f$ contracts each fiber of $g$. Thus the assertion follows immediately from \cite[Lemma 1.15(b)]{Deb01}. This implies (\ref{lem:flop1}).

Item (\ref{lem:flop2}) is an immediate consequence of Corollary \ref{prop:mds},  and is visualized in the following commuting diagram.
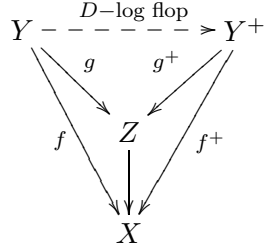
\[
\begin{xy}
  \xymatrix{
    Y\ar@{ -->}[rr]^{D-\text{log flop}} \ar[rd]^g\ar[rdd]_f && Y^+\ar[ld]_{g^+}\ar[ldd]^{f^+}\\
& Z\ar[d]&\\
& X &
}
\end{xy}
\]
The map $f^+$ is the new $\bQ$-factorialization which is obtained by the $D$-log flop.
\end{proof}

We finally come to the main result of this section. Roughly speaking, the following proposition asserts that for any effective Weil-divisor $D$ on $X$, there exists a $\bQ$-factorialization $f:Y\to X$ such that $(f^{-1}_*D)^\perp$ is in a sufficiently general position relative to the moving cone $\ol\NM_1(Y)$.

\begin{prop}\label{prop:transversal}
Let $(X,\Delta)$ be a dlt pair with $-(K_X+\Delta)$ ample, and let $D\neq 0$ be an effective $\bR$-Weil-divisor on $X$. Then there exists a $\bQ$-factorialization $(Y,\Delta_Y)$ such that the cone $\ol\NE_1(Y)_{K_Y+\Delta_Y=0}+\ol\NM_1(Y)$ has a $(K_Y+\Delta_Y)$-negative exposed ray which is not contained in $D_Y^\perp$, where $D_Y$ is the strict transform of $D$.
\end{prop}
The proof of Proposition \ref{prop:transversal} is quite long, and will be given in the following two Sections \ref{sec:prep} and \ref{sec:prooftrans}. 
\subsubsection{Preparation for the proof of Proposition \ref*{prop:transversal}}\label{sec:prep}
The proof of the Proposition consists of the following steps:
\begin{enumerate}
  \item Use log flops to construct the $\bQ$-factorialization, and
  \item prove that the $\bQ$-factorialization satisfies Proposition \ref{prop:transversal}.
\end{enumerate}
Since  the second part includes some tedious but not very challenging computations,  we divide these computations into the following two lemmas. 
The first lemma provides a criterion to decide whether a given ray in a cone is extremal, and can be formulated in terms of convex geometry, the second one analyzes the image of exposed moving rays via flips.
\begin{lem}[Criterion of extremeness]\label{lem:extremal}
  Let $V$ be a finite dimensional real vector space, and let $\sC^1,\sC^2\subset V$ be two closed, convex cones. Let $\alpha\in V^\vee$ be a linear form and $R\subset\sC^1$ a ray such that the following conditions hold.
\begin{itemize}
\item $R= \sC^1_{\alpha=0}$, and $\sC^1\subset\{\alpha\geq 0\}$,
\item $\sC^2\subset \{\alpha\geq 0\}$, and 
\item $R\not\subset\sC^2$ and $(-R)\not\subset\sC^2$.
\end{itemize}
Then $R$ is an extremal ray of $\sC^1+\sC^2$.
\end{lem}
\begin{proof}
 Observe that the set $\sD:=(\sC^1+\sC^2)_{\alpha=0}$ is an extremal face of $\sC^1+\sC^2$. Therefore, the face $\sD$ decomposes into 
\[
\sD= \sC^{1}_{\alpha=0} + \sC^{2}_{\alpha=0}=R+\sC^{2}_{\alpha=0}.
\]
Since $R\not\subset\sC_2$ and $(-R)\not\subset\sC^2$, it follows that $R$ is an extremal ray of $\sD$. To finish the proof, recall that being extremal is a transitive property, i.e., since $R$ is extremal in $\sD$  and $\sD$ is extremal in $\sC^1+\sC^2$, the ray $R$ is also extremal in $\sC^1+\sC^2$, as required.
\end{proof}

\begin{rem}
The ray $R\subset \sC^1+\sC^2$ is not necessarily exposed.  
\end{rem}

\begin{lem}[Flips of exposed rays]\label{lem:pfexposed}
Let $X,Y$ be $\bQ$-factorial normal projective varieties, and let 
$\varphi:X\rto Y$ be a birational map which is an isomorphism in codimension one.
Let $F\subset\ol\NM_1(X)$ be an exposed face, cut out by a pseudo-effective $\bR$-divisor $D$.
Then the image $\varphi_*(F)$ of $F$ via the numerical pushforward of curves is an exposed face of $\ol\NM_1(Y)$ which is cut out by $\varphi_*(D)$. 
\end{lem}

\begin{proof}
The assumptions imply that the vector spaces $N^1(X)_\bR$ and $N^1(Y)_\bR$ are isomorphic via the pullback and pushforward of divisors. Moreover, the restriction of the pushforward map to $\ol\NE^1(X)$ gives a bijection between the pseudo-effective cones $\ol\NE^1(X)$ and $\ol\NE^1(Y)$. By duality, the numerical pushforward and pullback of curve classes yields an isomorphism between $N_1(X)_\bR$ and $N_1(Y)_\bR$, and by \cite[Theorem 2.2]{BDPP04}, a bijection between $\ol\NM_1(X)$ and $\ol\NM_1(Y)$, in particular $\varphi_*(F)\subset\ol\NM_1(Y)$.
Since the divisor $D$ is pseudo-effective, its pushforward $\varphi_*(D)$ is pseudo-effective, as well.

It remains to prove that the equality $\varphi_*(D)^\perp\cap\ol\NM_1(Y)=\varphi_*(F)$ holds. This follows easily from the projection formula and the fact that pushforward and pullback are mutually inverse bijections. These computations are straightforward, thus we omit them.
\end{proof}
\begin{rem}
The lemma is also true for extremal faces, but becomes false if the map is not an isomorphism in codimension one, e.g., if $\varphi$ is a divisorial contraction.
\end{rem}

\subsubsection{Proof of Proposition \ref*{prop:transversal}}\label{sec:prooftrans}
  We start with an arbitrary $\bQ$-factorialization $f_0:Y_0\to X$. Set $\Delta_0:=\Delta_{Y_0}$, and let $D_0:=(f_0^{-1})_*D$ be the strict transform of the effective Weil divisor $D$ on $X$. Let $R_0$ be a $(K_{Y_0}+\Delta_{Y_0})$-negative extremal ray of the moving cone $\ol\NM_1(Y_0)$ which is not contained in $D_0^\perp$. By Corollary \ref{prop:mds}, the cone $\ol\NM_1(Y_0)$ is polyhedral, therefore $R_0$ is exposed and there is a pseudo-effective $\bR$-divisor $D_{R_0}$ such that 
\[
R_0=\ol\NM_1(Y_0)_{D_{R_0}=0}.
\]
Because of Corollary \ref{prop:mds} we can run the relative minimal model program for the pair $(Y_0,\Delta_0+D_{R_0})$ over $X$.  
Observe that this minimal model program only involves log flops and yields by Lemma \ref{lem:flop} a sequence of $\bQ$-factorializations of $X$. Because of Corollary \ref{prop:mds} we eventually obtain a minimal model over $X$ which is expressed in the following commutative diagram
\[
\begin{xy}
\xymatrix{
Y_0 \ar[rd]_{f_0} \ar@{ -->}[rr]^\varphi && Y\ar[ld]^f\\
& X, &
}
\end{xy}
\]
moreover the divisor $K_Y+\varphi_*(D_{R_0})+\varphi_*(\Delta_0)$ is $f$-nef. 

This finishes the construction of the $\bQ$-factorialization, and it remains to show that $Y$ has the required properties. To this end, we first observe that $Y_0,Y$, and $\varphi$ satisfy the conditions of Lemma \ref{lem:pfexposed}, hence the ray $R:=\varphi_*(R_0)$ is an exposed ray of $\ol\NM_1(Y)$, cut out by $D_R:=\varphi_*(D_{R_0})$. Moreover, since $K_Y+\varphi_*(\Delta_0)+D_R$ is $f$-nef and any $K_Y+\varphi_*(\Delta_0)$-trivial curve is contracted by $f$, we obtain the inclusion 
\[
\ol\NE_1(Y)_{K_Y+\varphi_*(\Delta_0)=0}\subset \{D_R\geq 0\}. 
\]
By Lemma \ref{lem:bignef}, the divisor  $K_Y+\varphi_*(\Delta_0)$ is big, thus $\ol\NE_1(Y)_{K_Y+\varphi_*(\Delta_0)=0}\cap\ol\NM_1(Y)=0$, and Lemma \ref{lem:extremal} applies. Altogether, the ray $R$ is an extremal ray of $\ol\NE_1(Y)_{K_Y+\varphi_*(\Delta_0)=0}+\ol\NM_1(Y)$. Since this cone is polyhedral by Corollary \ref{prop:mds}, the ray $R$ is even an exposed ray.

To finish the proof, we have to show that $R$ is not contained in the hyperplane $D_Y^\perp$, where $D_Y$ is the strict transform of $D$. Since $\varphi$ is an isomorphism in codimension one, the divisor $D_Y$ is also given by the pushforward of $D_0$ via $\varphi$. The projection formula immediately implies that $D_Y$ intersects any non-zero class $\gamma\in R$ positively, and the proof is finished.
\qed

\section{Families over log Fano varieties}\label{sec:isotrivial}
In this section we will prove the Isotriviality Theorem \ref{thm:isotrivial} by induction over the dimension. 
As a part of the induction we prove Theorem \ref{thm:moduli}, which is stated below. 
Assuming that Theorem \ref{thm:moduli} holds in dimension $n$, we first show that the family is necessarily isotrivial on certain moving curves, namely the curves we constructed in Theorem \ref{thm:mcone}. 
Next we show that for any proper algebraic subset $Z$ of $X$ there exists a moving curve that is not contained in $Z$ and intersects $Z$ properly. On this curve the family is isotrivial. This finally finishes the proof of the Isotriviality Theorem \ref{thm:isotrivial} for $n$-dimensional varieties. 

Assuming that Theorem \ref{thm:isotrivial} holds in dimension $n$ we will prove 
Theorem \ref{thm:moduli} in the $(n+1)$-dimensional case. This finally finishes the proof of both theorems in arbitrary dimension.

\subsection{A result of Kebekus and Kov{\'a}cs}
Given a smooth projective family of canonically polarized varieties, it is proved in \cite[Theorem 1.2]{KK10} that any run of the minimal model program for the base terminates with a Kodaira or Mori fiber space that factors the moduli map birationally, if the dimension of the base is less than or equal to three. A proof for surfaces can be found in \cite{KK08}. Since we discuss log Fano varieties, we will focus on the case of negative Kodaira-Iitaka dimension. As part of the induction we show that this result holds in arbitrary dimension. 

\begin{thm}[Moduli and the minimal model program, {\cite[Theorem 1.2]{KK10}}]\label{thm:moduli}
Let $(X,\Delta)$ be a $\bQ$-factorial dlt pair of negative Kodaira-Iitaka-dimension. Let $T\subset X$ be a subvariety of $\codim_X(T)\geq 2$ such that $X\setminus(T\cup\Supp\lrd\Delta\rrd)$ is smooth, and let $\mu: X\setminus(T\cup\Supp(\lrd\Delta\rrd))\to\fM$ be a map to the coarse moduli space of canonically polarized manifolds  which is induced  by a smooth projective family over $X\setminus(T\cup\Supp\lrd\Delta\rrd)$.

Then any terminating minimal model program $\lambda:X\rto X'$ leads to a Mori fiber space $\pi:X'\to B$ which factors the moduli map $\mu$ via $\pi\circ\lambda$ birationally. In other words, there exists a rational map $\nu: B\rto\fM$ such that the following diagram commutes.
\[
\begin{xy}
\xymatrix{
X \ar@{ -->}[rr]^\lambda \ar@{ -->}[d]_\mu  && X'\ar[d]^\pi \\
\fM && B.\ar@{ -->}_\nu[ll]
}
\end{xy}
\]
\end{thm}

\subsection{Proof of Theorems \ref*{thm:isotrivial} and \ref*{thm:moduli}}
\subsubsection{General strategy and setup}

The proof of the Theorems~\ref{thm:isotrivial} and \ref{thm:moduli} is by
induction on the dimension. For arbitrary $x$ the notation \emph{Theorem
  $x_n$} stands for ``Theorem $x$ in dimension at most $n$''.  The proof is
given in the following three steps.

\subsubsection*{Step 1: The case where the Picard number of $X$ is one} 

In this case, both theorems assert that a smooth family of canonically
polarized varieties is isotrivial over a logarithmic log Fano dlt pair with
Picard number one. A proof of this case is given in \cite{KK10} if $\dim X\leq
3$. It can be generalized to arbitrary dimension, since the Bogomolov-Sommese
vanishing for lc pairs holds in arbitrary dimensions, see \cite{GKKP10}. Note
that this case implies both theorems if $X$ is a curve.

\subsubsection*{Step 2: Theorem~\ref*{thm:moduli}$_n$ implies Theorem \ref*{thm:isotrivial}$_n$}

Assuming Theorem \ref{thm:moduli}$_n$ it follows from Proposition
\ref{prop:transversal} and Theorem \ref{thm:mcone} that the family is
isotrivial on ``sufficiently many'' moving curves. This implies Theorem
\ref{thm:isotrivial}$_n$.

\subsubsection*{Step 3: Theorem~\ref*{thm:isotrivial}$_n$ implies Theorem~\ref*{thm:moduli}$_{n+1}$}

Finally, we can apply Theorem \ref{thm:isotrivial}$_n$ to the general fiber of
a Mori fiber space, which in turn implies Theorem \ref{thm:moduli}$_{n+1}$.

\subsubsection{The case where the Picard number of $X$ is one}
To show that Theorem \ref{thm:moduli} holds if the Picard number is one, we have to use certain invertible sheaves $\sA\subset \Sym^n\Omega_X^1(\log \Delta)$ which were introduced by Viehweg and Zuo in \cite{VZ02}. These \emph{Viehweg-Zuo sheaves} are also discussed in \cite[Chapter 5]{KK10}. 

\begin{thm}[{\cite[Theorem 6.1]{KK10}}]\label{thm:picard}
  Let $(Z,\Delta)$ be log canonical logarithmic pair where $Z$ is projective
  $\bQ$-factorial. Assume that there exists a Viehweg-Zuo sheaf $\sA$ of
  positive Kodaira-Iitaka dimension, and that the divisor $-(K_Z+\Delta)$ is
  nef. Then the Picard number of $Z$ is greater than one.
\end{thm}
\begin{proof}
After replacing the old version of the Bogomolov-Sommese Vanishing Theorem \cite[Theorem 3.5]{KK10} with the new one \cite[Theorem 7.2]{GKKP10}, the proof given in \cite[Theorem 6.1]{KK10} applies verbatim for arbitrary dimension.
\end{proof}
\begin{lem}[Picard number one]\label{lem:picard}
Let $(X,\Delta)$ and $\mu$ as in Theorem \ref{thm:moduli} and assume that the the Picard number of $X$ is one.
Then $\mu$ is constant.
\end{lem}
\begin{proof}
Assume that $\mu$ is not constant. Since the Picard number is one, the $\bR$-divisor $\Delta$ is nef, in particular the pair $(X,\lrd\Delta\rrd)$ is dlt log Fano. Thus we can assume without loss of generality that $\Delta$ is reduced.

Let $\pi:\tilde X\to X$ be a log resolution of $(X,\Delta)$ such that $\pi^{-1}(T)$ is contained in the $\pi$-exceptional divisor $E\in\Div(\tilde X)$. Set $\tilde\Delta:=E+\pi^{-1}_*(\Delta)$, and note that $\tilde\Delta$ is snc and $\pi_*\tilde\Delta=\Delta$. Use $\pi$ to obtain a family of positive variation over $\tilde X\setminus\Supp\tilde\Delta$. It follows from \cite[Theorem 1.4]{VZ02} that there exists a Viehweg-Zuo sheaf $\tilde\sA\subset \Sym^n\Omega_{\tilde X}^1(\log\tilde\Delta)$ with $\kappa(\tilde\sA)>0$. Apply \cite[Lemma 5.2]{KK10} to obtain a Viehweg-Zuo sheaf $\sA\subset \Sym^n\Omega_X^1(\log \Delta)$ with $\kappa(\sA)\geq \kappa(\tilde\sA)>0$.
By Theorem \ref{thm:picard} the Picard number of $X$ is greater than one, which is a contradiction.
\end{proof}
Since curves always have Picard number one, we obtain the following
\begin{cor}[Start of induction, \protect{\cite[0.2]{Kov00}}]
Theorem~\ref{thm:moduli} and Theorem~\ref{thm:isotrivial} hold in dimension one. \qed
\end{cor}

\subsubsection{Theorem \ref*{thm:moduli}$_n$ implies Theorem \ref*{thm:isotrivial}$_n$}
We first use Theorem \ref{thm:moduli}$_{n}$ to show that a smooth family of canonically polarized varieties is isotrivial on certain moving curves.
\begin{prop}\label{prop:curvetriv}
Assume Theorem \ref{thm:moduli}$_n$. Let $(X,\Delta)$, $T$ and $\mu$ be as in Theorem \ref{thm:moduli}$_n$.  Let $R$ be a $(K_X+\Delta)$-negative exposed ray of the cone $\ol\NM_1(X)+\ol\NE_1(X)_{K_X+\Delta \geq 0}$. Let $H_R$ be the associated subset of the Hilbert scheme as in Theorem \ref{thm:mcone}. Then there exists a non-empty open subset $H_{R,\mu}$ of $H_R$ such that any curve $C\subset X$ that corresponds to a closed point of $H_{R,\mu}$ satisfies the following properties.
\begin{enumerate}
  \item The curve $C$ is not contained in $T\cup\Supp\lrd\Delta\rrd$, \label{item:general}
  \item the moduli map $\mu$ is constant on $C\cap \left(X\setminus (T\cup\Supp\lrd\Delta\rrd)\right)$, \label{item:constant}
  \item for any closed subset $Z\subset X$ of $\codim_X(Z)\geq 2$, there is a non-empty open subset $H^Z_{R,\mu}$ of $H_R$  such that any curve that corresponds to a closed point of $H^Z_{R,\mu}$ avoids $Z$. \label{item:codim2}
\end{enumerate}
 \end{prop}

\begin{proof}
We apply the Moving Cone Theorem~\ref{thm:mcone} and obtain an associated minimal model program $\lambda : X\rto X'$ and a Mori fibration $\pi: X'\to B$ such that any curve that corresponds to a point of $H_R$ is contained in the locus where $\lambda$ is well-defined and is mapped to a fiber of $\pi$. Theorem \ref{thm:moduli}$_n$ gives a commutative diagram of rational maps
\[
\begin{xy}
\xymatrix{
X \ar@{ -->}[rr]^\lambda \ar@{ -->}[d]_\mu  && X'\ar[d]^\pi \\
\fM && B,\ar@{ -->}[ll]
}
\end{xy}
\]
which becomes a diagram of morphisms on appropriate non-empty open sets. More precisely, let $V\subset B$ be the domain of $B\rto\fM$ and let $U'\subset X$ be the intersection of the domains of $\mu$ and $\lambda$. Then, if we set $U:=\lambda|_{U'}^{-1}(\pi^{-1}(V))$, we obtain the following commutative diagram of morphisms  
\[
\begin{xy}
\xymatrix{
U \ar[rr]^\lambda \ar[d]_\mu  && \pi^{-1}(V)\ar[d]^\pi \\
\fM && V.\ar[ll]
}
\end{xy}
\]
Let $A$ be a very ample divisor on $X$ in general position. Then the intersection $S:=(\Supp A\cap (X\setminus U))\subset X$ is a subvariety of $\codim_X(S)\geq 2$. Property~(\ref{item:codim}) of Theorem~\ref{thm:mcone} implies that there is an open subset $H^S_R$ of $H_R$ such that any closed point of $H^S_R$ corresponds to a curve that avoids $S$. 
We set $H_{R,\mu}:=H^S_R$, and it remains to show that $H_{R,\mu}$ has the required properties (\ref{item:general}), (\ref{item:constant}) and (\ref{item:codim2}).
Let $C\subset X$ be curve that corresponds to a closed point of $H_{R,\mu}$.

Since $A$ is chosen to be ample, $C$ intersects $A$ positively in a point $p\in\Supp A$. By definition, $p\notin X\setminus U$, which implies (\ref{item:general}).  

Since $C$ is not entirely contained in $X\setminus U$, the image $\pi\circ\lambda(C\cap U)$ is a point of $V$, thus the family is isotrivial on $C$. This implies (\ref{item:constant}).

To prove the last property (\ref{item:codim2}), recall that $H_R$ is irreducible and $H_{R,\mu}\subset H_R$ is open. For $Z\subset X$ of $\codim_X(Z)\geq 2$ let $H^Z_R$ be as in Property (\ref{item:codim}) of Theorem \ref{thm:mcone}. We set $H^Z_{R,\mu}:= H_{R,\mu}\cap H^Z_R$ which is non-empty and open in $H_{R,\mu}$. This implies (\ref{item:codim2}).
\end{proof}

\begin{lem}\label{lem:induction} Theorem \ref{thm:moduli}$_n$ implies Theorem \ref{thm:isotrivial}$_n$.
\end{lem}
\begin{proof}
  Let $(X,\Delta)$ and $T$ be as in Theorem~\ref{thm:isotrivial}, and that
  $\dim X = n$. Let $\fX\to X\setminus (T\cup\Supp\lrd\Delta\rrd)$ be a smooth
  projective family of canonically polarized manifolds.  As before, we denote
  by $\mu:X\rto\fM$ the induced moduli map to the coarse moduli space of
  canonically polarized manifolds. To prove that $\mu$ is constant we argue by
  contradiction and assume that this is not the case.  Since $\fM$ is
  quasi-projective, see \cite[Theorem 1.11]{Vie95}, we may choose a general
  hyperplane section $H$ on $\fM$. This is a divisor which intersects the
  image of $\mu$ properly, hence we can take its strict transform via $\mu$,
  denoted by $D_X\in\WDiv(X)$.  This is an effective Weil divisor to which we
  apply Proposition \ref{prop:transversal}. Accordingly, we obtain a
  $\bQ$-factorialization $f: Y \to X$ with boundary divisor $\Delta_Y :=
  f^{-1}_* \Delta$ and a $(K_Y+\Delta_Y)$-negative exposed ray $R$ of the cone
  $\ol\NM_1(Y)+\ol\NE_1(Y)_{K_Y+\Delta_Y\geq 0}$ which is not contained in the
  hyperplane $(f^{-1}_*(D_X))^\perp$ defined by the strict transform
  $D_Y:=f^{-1}_*(D_X)$. Observe that the family over
  $X\setminus(T\cup\Supp\lrd\Delta\rrd)$ can be pulled back along $f$ to a
  family over $Y\setminus(f^{-1}(T)\cup\Supp\lrd\Delta_Y\rrd)$, and the
  induced moduli map is given by $\mu_Y:=\mu\circ f$. Since $f$ is small, the
  set $f^{-1}(T)$ has codimension greater than or equal to two, thus the
  conditions of Proposition \ref{prop:curvetriv} are still satisfied.

 Consequently, we obtain a subset $H_{R,\mu_Y}$ of the Hilbert scheme such that $\mu_Y$ is constant on any curve $C$ in $H_{R,\mu_Y}$. 
   Denote by $S\subset \Supp D_Y$ the set of points where the moduli map $\mu_Y$ is not defined. Since $\codim_YS\geq 2$ and because of Property~(\ref{item:codim2}) of Proposition~\ref{prop:curvetriv}, there is an open subset $H^S_{R,\mu_Y}$ of $H_{R,\mu_Y}$ such that the curves that correspond to this subset avoid $S$. Moreover, if $A$ is a very ample divisor in general position on $Y$, then we can assume, after shrinking $H_{R,\mu_Y}$ if necessary, that any such curve avoids $(\Supp A)\cap(\Supp D_Y)$. In particular, any curve that corresponds to a closed point of $H_{R,\mu_Y}$ is not entirely contained in $\Supp D_Y$.

Let $C$ be an arbitrary curve that corresponds to a closed point of $H^S_{R,\mu_Y}$. Due to Proposition \ref{prop:curvetriv}, the image of $C$ is a point in $p\in\fM$. Since $C$ intersects $D_Y$ outside $S$, this point $p$ is an element of the hyperplane section $H$ which in turn implies that $C$ is contained in $D_Y$. This finally contradicts the choice of $C$.
\end{proof}
\begin{rem} \label{rem:ample}
Note that the assumption that $(X,\Delta)$ is log Fano is only needed to apply Proposition \ref{prop:transversal}. More precisely, the proof of Theorem \ref{thm:isotrivial} still works if we assume that Proposition \ref{prop:transversal} holds for the pair $(X,\Delta)$, instead of assuming that $(X,\Delta)$ log Fano.
\end{rem}

\subsubsection{Theorem \ref*{thm:isotrivial}$_n$ implies Theorem \ref*{thm:moduli}$_{n+1}$, end of proof}
To finish the proof, we show the following
\begin{lem}
Theorem~\ref{thm:isotrivial}$_n$ implies Theorem~\ref{thm:moduli}$_{n+1}$.
\end{lem}
\begin{proof}
  Let $\lambda:X\rto X'$ be a minimal model program which leads to a Mori
  fiber space $\pi:X'\to B$. Set $\Delta':=\lambda_*\Delta$, and let $T'$ be
  the union of the indeterminacy locus of $\lambda^{-1}$ and the closure of
  the image of $T$. Note that $\codim_{X'}T'\geq 2$ holds. We use
  $\lambda^{-1}$ to pull the family back to a family $f':Y'\to
  X'\setminus(\Supp\lrd\Delta'\rrd\cup T')$. Then we have to show that the
  family is isotrival on a general fiber of $\pi$.

  If the Picard number $\rho(X')$ of $X'$ is one, then $(X',\Delta')$ is in
  particular log Fano. In this case Lemma \ref{lem:picard} implies the
  assertion.
 
  Otherwise, if $\rho(X')>1$, then $\dim B\geq 1$. Let $F$ be a general fiber
  of $\pi$, then $(F,\Delta'|_F)$ is dlt log Fano. Moreover, $\codim_F(F\cap
  T')\geq 2$, and $\lrd\Delta'|_F\rrd = \lrd\Delta'\rrd|_F$. Since $\dim F\leq
  n$, Theorem \ref{thm:isotrivial}$_n$ implies that the family restricted to
  $F$ is isotrivial, which finishes the proof.
\end{proof}

\section{A corollary of Theorem~\ref*{thm:isotrivial}}\label{sec:corollary}
We are now able to discuss some properties of the cone
\[
\ol\NM_1(X)+\ol\NE_1(X)_{K_X+\Delta \geq 0}.
\]
First we recall some well-known facts.
\begin{fact}[\protect{\cite[Theorem 1.1]{Leh09} and \cite[Corollary 1.35]{BCHM10}}]
Let $(X,\Delta)$ be a $\bQ$-factorial dlt pair, 
Then the following holds.
\begin{itemize}
  \item If $-(K_X+\Delta)$ is ample, then $\ol\NM_1(X)$ is a rational polyhedron.
  \item More generally, there are countably many rays $(R_i)_{i\in\bN}\subset \ol\NM_1(X)$ such that 
\[
\ol\NM_1(X)+\ol\NE_1(X)_{K_X+\Delta \geq 0} = \ol\NE_1(X)_{K_X+\Delta \geq 0}+\sum_i R_i.
\]
These rays are locally discrete away from hyperplanes that support both $\ol\NE_1(X)_{K_X+\Delta\geq 0}$ and $\ol\NM_1(X)$.
\end{itemize}
\end{fact}If $(X,\Delta)$ is a pair that admits a family of positive variation we can apply our proof of Theorem \ref{thm:isotrivial} to obtain another result. Remark \ref{rem:ample} implies that Proposition \ref{prop:transversal} cannot hold for $(X,\Delta)$. This in turn implies the following observation.
\begin{obs}
If $(X,\Delta)$ is a dlt pair that admits a non-isotrivial family, then Proposition \ref{prop:transversal} does not hold for $(X,\Delta)$. In particular, if $X$ is $\bQ$-factorial, then there is a hyperplane $H\subset N_1(X)$ such that any $(K_X+\Delta)$-negative exposed ray of $\ol\NM_1(X)+\ol\NE_1(X)_{K_X+\Delta\geq 0}$ is contained in $H$.
\end{obs}

\end{document}